\documentclass[12pt]{amsart}
\usepackage{amsmath}
\usepackage{amsthm}
\usepackage{amssymb}
\usepackage[T1]{fontenc}
\usepackage[sc]{mathpazo}
\usepackage{enumerate}
\usepackage{hyperref}
\usepackage{etoolbox}

\patchcmd{\section}{\scshape}{\bfseries}{}{}

\newenvironment{nouppercase}{%
  \renewcommand{\uppercasenonmath}[1]{}}{}

\newcommand{\C}{{\mathbb{C}}}
\newcommand{\R}{{\mathbb{R}}}

\newcommand{\Z}{{\mathbb{Z}}}

\DeclareMathOperator{\Log}{Log}

\newtheorem{theorem}{Theorem}[section]

\newtheorem{lemma}[theorem]{Lemma}

\theoremstyle{definition}
\newtheorem{definition}[theorem]{Definition}

\theoremstyle{remark}
\newtheorem{remark}[theorem]{Remark}
\newtheorem{example}[theorem]{Example}

\newcommand{\To}{\Rightarrow}
\renewcommand{\subset}{\subseteq}

\newcommand{\m}{\mathfrak{m}}
\newcommand{\p}{\mathfrak{p}}
\newcommand{\x}{\mathsf{x}}
\newcommand{\z}{\mathsf{z}}

\renewcommand{\P}{\mathbb{P}}
\newcommand{\Sper}{\mathrm{Sper}}
\newcommand{\conj}[1]{\overline{#1}}

\begin{document}
\title[~]
{A Positivstellensatz for forms on the positive orthant}
\author[~]
{Claus Scheiderer and Colin Tan}

\address
{Claus Scheiderer,
Fachbereich Mathematik und Statistik,
Universt\"at Konstanz,
Konstanz 78457,
Germany}
\email{claus.scheiderer@uni-konstanz.de}

\address
{Colin Tan,
Department of Statistics \& Applied Probability,
National University of Singapore,
Block S16,
6 Science Drive 2,
Singapore 117546}
\email{statwc@nus.edu.sg}

\begin{nouppercase}
\maketitle
\end{nouppercase}
\numberwithin{equation}{section}

\begin{abstract}
Let $p$ be a nonconstant form in $\R[x_1,\dots,x_n]$ with
$p(1,\dots,1)>0$. If $p^m$ has strictly positive coefficients for
some integer $m\ge1$, we show that $p^m$ has strictly positive
coefficients for all sufficiently large~$m$.
More generally, for any such $p$, and any
form $q$ that is strictly positive on $(\R_+)^n\setminus\{0\}$, we
show that the form $p^mq$ has strictly positive coefficients for all
sufficiently large~$m$. This result can be considered as a strict
Positivstellensatz for forms relative to $(\R_+)^n\setminus\{0\}$.
We give two proofs, one based on results of Handelman, the other on
techniques from real algebra.
\end{abstract}


\section{Introduction} \label{sec: intro}

Given polynomials $p, q \in \R[\x] := \R[x_1, \dots, x_n]$
where all the coefficients of $p$ are nonnegative, Handelman
\cite{Handelman86} gave a necessary and sufficient condition
(reproduced as Theorem
\ref{thm: HandelmanCharacterization} below) for there to exist a
nonnegative integer $m$ such that the coefficients of $p^m q$ are all
nonnegative. In another paper \cite{Handelman92}, Handelman showed
that, given a polynomial $p \in \R[\x]$ such that
$p(1, \dots, 1) > 0$, if the coefficients of $p^m$ are all
nonnegative for some $m > 0$, then the coefficients of $p^m$ are all
nonnegative for every sufficiently large $m$.

In the case where $p$ is a form (i.e.\ a~homogeneous polynomial),
there is a stronger positivity condition that $p$ may satisfy.
If $p(x) = \sum_{|w| = d} a_wx^w \in \R[\x]$ is homogeneous
of degree $d$ (with $a_w\in\R$), we say that $p$ {\emph{has strictly
positive coefficients}} if $a_w > 0$ for all $|w| = d$.
Here we use standard multi-index notation, where an $n$-tuple
$w = (w_1,\ldots, w_n)$ of nonnegative integers has length
$|w| := w_1 + \cdots + w_n$ and	$x^{w} = x_1^{w_1} x_2^{w_2}
\cdots x_n^{w_n}$.	
Denote the closed positive orthant of real $n$-space by
$$\R_+^n\>:=\>\{x=(x_1,\dots,x_n)\in\R^n\colon x_1,\dots,x_n
\ge0\}.$$
Our main result is as follows:

\begin{theorem} \label{thm: strictPositivstellensatz}
Let $p \in \R[\x]$ be a nonconstant real form.
The following are equivalent:
\begin{enumerate}[(A)]
\item \label{thm: someOddPos}
The form $p^m$ has strictly positive coefficients for some odd
$m\ge1$.
\item \label{thm: somePos}
The form $p^m$ has strictly positive coefficients for some $m\ge1$,
and $p(x) > 0$ at some point $x \in \R_+^n$.
\item \label{thm: eventualPos}
For each real form $q \in \R[\x]$ strictly positive on $\R_+^n
\setminus\{0\}$, there exists a positive integer $m_0$ such that
$p^m q$ has strictly positive coefficients for all $m \ge m_0$.
\end{enumerate}
\end{theorem}

Theorem \ref{thm: strictPositivstellensatz} can be derived from an
isometric imbedding theorem for holomorphic bundles, due to
Catlin-D'Angelo \cite{CD99}. The argument is sketched in an
appendix at the end of this paper. Another condition equivalent to
each of the three conditions Theorem
\ref{thm: strictPositivstellensatz} was given by the second author
and To in \cite{TanTo}. The line of argumentation in \cite{TanTo} is
analytic in nature, and the proof therein invokes Catlin-D'Angelo's
isometric embedding theorem.

As the statement of Theorem \ref{thm: strictPositivstellensatz}
involves only real polynomials, it is desirable to give a purely
algebraic proof, which is what we shall do below. We will in fact
give two proofs of very different nature. Both are
independent of Catlin-D'Angelo's proof in \cite{CD99}, which uses
compactness of the von Neumann operator on pseudoconvex domains of
finite-type domains in $\C^n$ and an asymptotic expansion of the
Bergman kernel function by Catlin \cite{Catlin99}. We remark that in
the case when $n = 2$, Theorem \ref{thm: strictPositivstellensatz}
follows from De~Angelis' work in \cite{deAngelis03} and has been
independently observed by Handelman \cite{HandelmanMO}.

Our first proof of Theorem \ref{thm: strictPositivstellensatz} uses
the criterion of Handelman \cite{Handelman86} mentioned above.
Our second proof reduces Theorem
\ref{thm: strictPositivstellensatz} to the archimedean local-global
principle due to the first author, in a spirit similar to
\cite{Scheiderer12}.

For a real form, having strictly positive coefficients is a
certificate for being strictly positive on $\R_+^n\setminus\{0\}$.
Therefore Theorem \ref{thm: strictPositivstellensatz} can be seen as
a Positivstellensatz for forms $q$, relative to $\R_+^n\setminus
\{0\}$. In particular, the case where $p = x_1 + \cdots + x_n$
specializes to the classical P\'olya Positivstellensatz
\cite{Polya28} (reproduced in \cite[pp. 57--60]{HLP52}).
For any $n \ge 2$ and even $d \ge 4$, there are examples of
degree-$d$ $n$-ary real forms $p$ with some negative coefficient
that satisfy the equivalent conditions of Theorem
\ref{thm: strictPositivstellensatz} (see Example \ref{eq: DV} below).
\medbreak

\paragraph{{\textbf{Acknowledgements.}}}
The second author would like to thank his PhD supervisor Professor
Wing-Keung To for his continued guidance and support. We would also
like to thank David Handelman for his answer on MathOverflow
\cite{HandelmanMO}, and the anonymous referee for pointing out
reference \cite{kr}.


\section{A theorem of Handelman}

Let $p = \sum_{w \in \Z^n} c_w x^w\in\R[\x,\x^{-1}]
:= \R[x_1, \ldots, x_n, x_1^{-1}, \ldots, x_n^{-1}]$ be a Laurent
polynomial. Following Handelman \cite{Handelman86} we introduce the
following terminology. The {\emph{Newton diagram}} of $p$ is the set
$\Log(p) := \{w \in \Z^n : c_w \ne0\}$. A subset $F$ of $\Log(p)$ is
a \emph{relative face} of $\Log(p)$ if there exists a face
$K$ of the convex hull of $\Log(p)$ in $\R^n$ such that
$F = K \cap \Log(p)$.
In particular, the subset $\Log(p)$ is itself a relative face of
$\Log(p)$, called the \emph{improper relative face}. Given a set
$F\subset\Z^n$, an integer $k\ge1$ and a point $z\in\Z^n$, we write
$kF+z:=\{w^{(1)}+\cdots+w^{(k)}+z\colon w^{(1)},\dots,w^{(k)}\in
F\}\subset\Z^n$. For a subset $E$ of $\Z^n$ and the above Laurent
polynomial $p$ we write $p_E:=\sum_{w\in E}c_wx^w$.

\begin{definition}\label{dfnstratum}
Let $p\in\R[\x,\x^{-1}]$ be a nonzero Laurent polynomial. Given a
relative face $F$ of
$\Log(p)$ and a finite subset $S$ of $\Z^n$, a {\emph{stratum of $S$
with respect to $F$}} is a nonempty subset $E\subset S$ such that
\begin{enumerate}[(i)]
\item\label{cond: stratumOne}
there exist $k\ge1$ and $z\in\Z^n$ such that $E\subset kF+z$;
and
\item\label{cond: stratumTwo}
whenever $E\subset kF+z$ for some $z\in\Z^n$ and some $k\ge1$, it
follows that $E=(kF+z)\cap S$.
\end{enumerate}
A stratum $E$ of $S$ with respect to $F$ is \emph{dominant} if, in
addition, the following holds:
\begin{enumerate}[(i)]
\setcounter{enumi}{2}
\item \label{cond: stratumDominant}
If $E\subset(k\Log(p)+z)\setminus(kF+z)$ for some $k\ge1$ and some
$z\in\Z^n$, then $(kF+z)\cap S=\emptyset$.
\end{enumerate}
\end{definition}

\begin{theorem}[Handelman {\cite[Theorem A]{Handelman86}}]
  \label{thm: HandelmanCharacterization}
Let $p$ and $q$ be Laurent polynomials in $\R[\x,\x^{-1}]$, where $p$
has nonnegative coefficients. Then $p^m q$ has nonnegative
coefficients for some positive integer $m$ if, and only if, both the
following conditions hold:
\begin{enumerate}[\indent(a)]
\item \label{cond: HandelmanConditionOne}
For each dominant stratum $E$ of $\Log(q)$ with respect to the
improper relative face $\Log(p)$, the polynomial $q_{E}$ is strictly
positive on the interior of $\R_+^n$.
\item \label{cond: HandelmanConditionTwo}
For each proper relative face $F$ of $\Log(p)$, and each dominant
stratum $E$ of $\Log(q)$ with respect to $F$, there exists a positive
integer $m$ such that $p_F^m  q_E$ has nonnegative coefficients.
\end{enumerate}
\end{theorem}

\noindent
Here, for a Laurent polynomial $f$, by ``$f$ has nonnegative
coefficients'', we mean that all coefficients of $f$
are nonnegative. As observed in \cite{Handelman86},
the product of a suitable monomial with $p_F$ (resp. $f_E$) is a
Laurent polynomial involving fewer than $n$ variables (when $F$ is
proper), so that the condition \eqref{cond: HandelmanConditionTwo} is
inductive.


\section{First proof of Theorem \ref{thm: strictPositivstellensatz}}
    \label{sec: proofOfMainTheorem}

We fix an integer $n\ge1$ and use the notation $\x=(x_1,\dots,x_n)$
and $[n]=\{1,\dots,n\}$. Given $z\in\Z^n$ and a subset $J$ of $[n]$,
let $z_J:=(z_j)_{j\in J}\in\Z^J$ denote the corresponding
truncation of~$z$. For a nonnegative integer $d$, we write
$(\Z^n_+)_d=\{w\in\Z^n\colon
w_1\ge0,\dots,w_n\ge0$, $w_1+\cdots+w_n=d\}$.

\begin{lemma}\label{domstrat}
Let $p\in\R[\x]$ be a form of degree $d\ge1$ with strictly positive
coefficients. Let $e\ge0$, and let $S\subset(\Z^n_+)_e$ be a nonempty
subset.
\begin{itemize}
\item[(a)]
The relative faces of $\Log(p)$ are the sets $F_J:=\{w\in (\Z^n_+)_d
\colon w_J=0\}$, where $J\subset[n]$ is a subset.
\item[(b)]
Let $J\subset [n]$. For each stratum $E$ of $S$ with respect to
$F_J$, there exists $\beta\in \Z_+^J$ satisfying $|\beta|\le e$
such that
$$E\>=\>E_{J,\beta}\>:=\>\{w\in S\colon w_J=\beta\}$$
\item[(c)]
If $S=(\Z_+^n)_e$ and $\emptyset\ne J\subsetneq[n]$, the stratum
$E_{J,\beta}$ of $S$ with respect to $F_J$ is dominant if and only if
$\beta=0$.
\end{itemize}
\end{lemma}

In particular, $E=S$ is the only stratum of $S$ with respect to the
improper relative face $\Log(p)$ of $\Log(p)$, by (b). Note that this
stratum is dominant for trivial reasons.

\begin{proof}
By assumption we have $\Log(p)=(\Z^n_+)_d$. Denote this set by $F$.
Assertion (a) is clear. Note that $J=\emptyset$ resp.\ $J=[n]$
gives $F_J=F$ resp.\ $F_J=\emptyset$.
To prove (b), fix a subset $J\subset[n]$, and let
$E\subset S$ be a stratum of $S$ with
respect to $F_J$. So there exist $k\ge1$ and $z\in\Z^n$ such that
$E=(kF_J+z)\cap S$. By the particular shape of $F$ we have
$$kF_J+z\>=\>\{w\in\Z^n\colon|w|=ke+|z|,\ w\ge z,\ w_J=z_J\},$$
where $w\ge z$ means $w_i\ge z_i$ for $i=1,\dots,n$.
Therefore $E\subset E_{J,\beta}$ with $\beta:=z_J$.
Note that $E_{J,\beta}$ can be nonempty only when $|\beta|\le e$.
The proof of (b) will be completed if we show that
$E_{J,\beta}\subset lF_J+y$ holds for suitable $l\ge1$ and
$y\in\Z^n$.
To this end it suffices to observe that there exist $l\ge1$ and
$y\in\Z^n$ such that $ld\ge e$, $|y|=e-ld$, $y_J=\beta$ and
$y_i\le0$ for $i\in[n]\setminus J$.
These $l$ and $y$ will do the job.

It remains to prove (c), so assume now that $S=(\Z^n_+)_e$. Let
$J\subsetneq [n]$ be a proper subset, and let $\beta\in\Z_+^J$ be
such that $E_{J,\beta}$ is nonempty (hence a stratum of $S$). First
assume $\beta\ne0$. There exist $k\ge1$ and $z\in\Z^n$ such that
$0\le z_J\le\beta$ and $z_J\ne\beta$, such that $z_{[n]\setminus J}
\le0$, and such that $|z|=e-kd$.
Then $E_{J,\beta}\subset kF+z$
and $E_{J,\beta}\cap(kF_J+z)=\emptyset$.
But there exists $w\in S$ with $w_J=z_J$, showing that $(kF_J+z)
\cap S\ne\emptyset$, whence $E_{J,\beta}$ is not dominant. On the
other hand, $E_{J,\beta}$ is easily seen to be dominant when
$\beta=0$.
\end{proof}

We now give a first proof of Theorem
\ref{thm: strictPositivstellensatz}. The implications
\eqref{thm: someOddPos} $\To$ \eqref{thm: somePos} and
\eqref{thm: eventualPos} $\To$ \eqref{thm: someOddPos} are trivial.
To prove \eqref{thm: somePos} $\To$ \eqref{thm: eventualPos}, it
suffices to show the following apparently weaker statement:

\begin{lemma}\label{weakerstatement}
Given forms $f,\,g\in\R[x_1,\dots,x_n]$, where $f$ is nonconstant
with strictly positive coefficients and where $g$ is strictly
positive on $\R^n_+\setminus\{0\}$, there exists $l\ge1$ such that
$f^lg$ has nonnegative coefficients.
\end{lemma}

Assuming that Lemma \ref{weakerstatement} has been shown, we can
immediately state a stronger version of this lemma. Namely, under the
same assumptions it follows that $f^lg$ actually has strictly positive
coefficients for suitable $l\ge1$. Indeed, choose a form $g'$ with
$\deg(g')=\deg(g)$ such that $g'$ has strictly positive coefficients
and the difference $h:=g-g'$ is strictly positive on
$\R_+^n\setminus\{0\}$, for instance $g'=c(x_1+\cdots+x_n)^{\deg(g)}$
with sufficiently small $c>0$. Applying Lemma \ref{weakerstatement} to
$(f,h)$ instead of $(f,g)$ gives $l\ge1$ such that $f^lh$ has
nonnegative coefficients. Since $f^lg'$ has strictly positive
coefficients, the same is true for $f^lg=f^lg'+f^lh$.

Now assume that condition \eqref{thm: somePos} of Theorem
\ref{thm: strictPositivstellensatz} holds. Then the form $p$ is
strictly positive on $\R_+^n\setminus\{0\}$.
In order to prove \eqref{thm: eventualPos}, apply the strengthened
version of Lemma \ref{weakerstatement} to $(f,g)=(p^m,\,p^ih)$ for
$0\le i\le m-1$. This gives $l\ge1$ such that $p^{lm+i}h$ has
nonnegative coefficients for all $i\ge0$, which is
\eqref{thm: eventualPos}. So indeed it suffices to prove Lemma
\ref{weakerstatement}.

\begin{proof}[Proof of Lemma \ref{weakerstatement}]
The case $n = 1$ is trivial.
Suppose that $n>1$ and the above statement holds in less than $n$
variables.
Let $\deg(f)=d\ge1$ and $\deg(g)=e$. As before, choose a form $g'$
with $\deg(g')=e$ and with strictly positive coefficients such that
$h:=g-g'$ is strictly positive on $\R^n_+\setminus\{0\}$. This can
be done in such a way that $\Log(h)=(\Z^n_+)_e$, i.e.\ all
coefficients of $h$ are nonzero.

We shall verify that the pair $(f,h)$ satisfies the conditions in
Theorem \ref{thm: HandelmanCharacterization}. Since $f$ has strictly
positive coefficients, the only (dominant) stratum of $S=\Log(h)$
with respect to $F=\Log(f)$ is $E=S$, by Lemma \ref{domstrat}. Thus
$h_E=h$ is strictly positive on $\R_+^n\setminus\{0\}$,
so that condition \eqref{cond: HandelmanConditionOne} is satisfied.
Next, let $J\subset[n]$ be a proper nonempty subset. Using the
notation of Lemma \ref{domstrat}, the only dominant stratum of
$S=\Log(h) = (\Z_+^n)_e$ with respect to the proper relative face
$F_J$ of $F=\Log(f)$ is $E:=E_{J,0}=\{w\in S\colon w_J=0\}$,
according to Lemma \ref{domstrat}(c). Without loss of generality we
may assume $J=\{r+1,\dots,n\}$ for some $1\le r<n$, where $J$ has
cardinality $n-r$. Then $h_E$ is a form in $\R[x_1,\dots,x_r]$ that
is strictly positive on $\R_+^r\setminus\{0\}$, since
$h_E(x_1,\dots,x_r)=h(x_1,\dots,x_r,0,\dots,0)>0$ for all
$(x_1,\dots,x_r)\in\R_+^r\setminus\{0\}$. Moreover, $f_{F_J}$ is a
form in $\R[x_1,\dots,x_r]$ with strictly positive coefficients. By
the inductive hypothesis there exists $m\ge1$ such that all
coefficients of $(f_{F_J})^mh_E$ are nonnegative, which shows that
$(f,h)$ satisfies condition \eqref{cond: HandelmanConditionTwo}.
Therefore, by Theorem \ref{thm: HandelmanCharacterization},
there exists $l\ge1$ such that $f^lh$ has nonnegative coefficients.
\end{proof}


\section{Archimedean local-global principle for semirings}
    \label{sec:archlgpsemiring}

Let $A$ be a (commutative unital) ring, and let $T\subset A$ be a
subsemiring of $A$, i.e.\ a subset containing $0,\,1$ and closed
under addition and multiplication. Recall that $T$ is said to be
\emph{archimedean} if for any $f\in A$ there exists $n\in\Z$ with
$n+f\in T$, i.e.\ if $T+\Z=A$.
The real spectrum $\Sper(A)$ of $A$ (see e.g.\ \cite{bcr}
7.1, \cite{ma} 2.4) can be defined as the set of all pairs
$\alpha=(\p,\le)$ where $\p$ is a prime ideal of $A$ and $\le$ is an
ordering of the residue field of~$\p$.
Given a semiring $T\subset A$, let $X_A(T)\subset\Sper(A)$ be
the set of all $\alpha\in\Sper(A)$ such that $f\ge_\alpha0$ for every
$f\in T$. We say that $f\in A$ satisfies $f\ge0$ (resp.\ $f>0$) on
$X_A(T)$ if $f\ge_\alpha0$ (resp.\ $f>_\alpha0$) for every $\alpha\in
X_A(T)$.

We recall the archimedean Positivstellensatz in the following form.
In a weaker form, this result was already proved by Krivine
\cite{kr}.

\begin{theorem}[{\cite{sw}} Corollary~2] \label{archpss}
Let $A$ be a ring, and let $T\subset A$ be an archimedean semiring
containing $\frac1n$ for some integer $n>1$. If $f\in A$ satisfies
$f>0$ on $X_A(T)$, then $f\in T$.
\end{theorem}

We will need to apply Theorem \ref{archlgpsemiring} below, which is a
local-global principle for archimedean semirings. A slightly weaker
version of this result was already proved in \cite{bss} Theorem 6.5.
We give a new proof which is considerably shorter than the proof in
\cite{bss}.

\begin{theorem}\label{archlgpsemiring}
Let $A$ be a ring, let $T\subset A$ be an archimedean semiring
containing $\frac1n$ for some integer $n>1$, and let $f\in A$. Assume
that for any maximal ideal $\m$ of $A$ there exists an element $s\in
A\setminus\m$ such that $s\ge0$ on $X_A(T)$ and $sf\in T$. Then
$f\in T$.
\end{theorem}

\begin{proof}
There exists an integer $k\ge1$ and elements $s_1,\dots,s_k\in A$
with $\langle s_1,\dots,s_k\rangle=\langle1\rangle$, and with
$s_if\in T$ and $s_i\ge0$ on $X_A(T)$ for $i=1,\dots,k$. By
\cite{sch:surf} Prop.\ 2.7 there exist $a_1,\dots,a_k\in A$ with
$\sum_{i=1}^ka_is_i=1$ and with $a_i>0$ on $X_A(T)$ ($i=1,\dots,k$).
Since $T$ is archimedean, the last condition implies $a_i\in T$, by
the Positivstellensatz \ref{archpss}. It follows that $f=\sum_{i=1}^k
a_i(s_if)\in T$.
\end{proof}


\section{Second proof of Theorem \ref{thm: strictPositivstellensatz}}
    \label{sec: 2ndproofOfMainTheorem}

As in the first proof, it suffices to prove Lemma
\ref{weakerstatement}. So let $f\in\R[\x]=\R[x_1,\dots,x_n]$ be a
form of degree $\deg(f)=d\ge1$ with strictly positive coefficients,
say $f=\sum_{|\alpha|=d}c_\alpha x^\alpha$. Let $S\subset\R[\x]$ be
the semiring consisting of all polynomials with nonnegative
coefficients. We shall work with the ring
$$A\>=\>\Bigl\{\frac p{f^r}\colon r\ge0,\ p\in\R[\x]_{dr}\Bigr\}$$
of homogeneous fractions of degree zero, considered as a subring of
the field $\R(\x)$ of rational functions. Let $V\subset\P^{n-1}$ be
the complement of the projective hypersurface $f=0$. Then $V$ is an
affine algebraic variety over $\R$, with affine coordinate ring
$\R[V]=A$. As a ring, $A$ is generated by $\R$ and by
the fractions $y_\alpha=\frac{x^\alpha}g$ where $|\alpha|=d$. Let $T$
be the subsemiring of $A$ generated by $\R_+$ and by the $y_\alpha$
($|\alpha|=d$). So the elements of $T$ are precisely the fractions
$\frac p{f^r}$, where $r\ge0$ and $p\in S$ is homogeneous of
degree~$dr$.

The semiring $T$ is archimedean, as follows from the identity
$\sum_{|\alpha|=d}c_\alpha y_\alpha=1$ and from $c_\alpha>0$ for all
$\alpha$ (\cite{BerrWormann} Lemma~1). First we prove Lemma
\ref{weakerstatement} under an extra condition.

\begin{lemma} \label{lem: assumeDegreeCondition}
Let $f,\,g\in\R[x_1,\dots,x_n]$ be forms where $f$ is nonconstant
with strictly positive coefficients and $g$ is strictly positive on
$\R^n_+\setminus\{0\}$. If $\deg(f)$ divides $\deg(g)$, there exists
$m\ge1$ such that $f^mg$ has nonnegative coefficients.
\end{lemma}

\begin{proof}
Suppose that $r$ is a positive integer such that $\deg(g)=r\deg(f)$.
Then the fraction $\frac g{f^r}$ lies in $A$ and is strictly positive
on $X_A(T)$, since $g$ is positive on $\R^n_+$. Hence the archimedean
Positivstellensatz (Theorem \ref{archpss}) gives $\frac g{f^r}\in T$,
and clearing denominators we get the desired conclusion.
\end{proof}

\begin{remark}
When $\deg(f)=1$, Lemma \ref{lem: assumeDegreeCondition} is in fact
P\'olya's Positivstellensatz \cite{Polya28}. In this case, our proof
above becomes essentially the same as the proof of \cite{Polya28}
given by Berr and W\"ormann in \cite{BerrWormann}.
\end{remark}

For the general case when $\deg(f)$ does not necessarily divide
$\deg(g)$, we need a more refined argument as follows. It is similar
to the approach in \cite{Scheiderer12}.

\begin{proof}
[Proof of Lemma \ref{weakerstatement}]
Fix integers $k\ge0$, $r\ge0$ such that $k+e=dr$, and consider the
fraction $\varphi:=\frac{x_1^kg}{f^r}\in A$. It suffices to show
$\varphi\in T$. Indeed, this means that there are $s\ge0$ and
$p\in S$, homogeneous of degree $ds$, such that $\varphi=\frac
p{f^s}$. We may assume $s\ge r$, then $f^{s-r}x_1^kg$ has
nonnegative coefficients. Clearly this implies that $f^{s-r}g$ has
nonnegative coefficients.

We prove $\varphi\in T$ by applying the local-global principle
\ref{archlgpsemiring}. So let $\m$ be a maximal ideal of $A$. Then
$\m$ corresponds to a closed point $z$ of the scheme $V$, and hence
of $\P^{n-1}$. There exist real numbers $t_1,\dots,t_n>0$ such that
the linear form $l=\sum_{i=1}^nt_ix_i$ does not vanish in $z$.
Hence the element $\psi:=\frac{l^d}f$ of $A$ does not lie in $\m$. On
the other hand, $\psi>0$ on $X_A(T)$, since $l$ and $f$ are strictly
positive on $\R^n_+\setminus\{0\}$.

By Lemma \ref{lem: assumeDegreeCondition}, applied to $l$ and $g$,
there exists an integer $N\ge1$ for which $l^Ng\in S$.
Choose an integer $m\ge1$ so large that $md\ge N$. Then
$$\psi^m\varphi\>=\>\frac{l^{md}x_1^kg}{f^{m+r}}$$
lies in $T$. From Theorem \ref{archlgpsemiring} we therefore deduce
$\varphi\in T$, as desired.
\end{proof}

We conclude with an example, as promised in the introduction.

\begin{example} \label{eq: DV}
For $n \ge 2$ and even $d = 2k \ge 4$, the form
$$p_{\lambda} = (x_1+x_2)^{2k} - \lambda x_1^kx_2^k +
\sum_{\stackrel{|w| = 2k}{w_i \neq 0 {\text{ for some } i \ge 3}}}
x^w\>\in\R[x_1,\dots,x_n]$$
of degree $d$
satisfies the equivalent conditions of Theorem
\ref{thm: strictPositivstellensatz} and has a negative coefficient
(of the monomial $x_1^kx_2^k$) whenever $\binom{2k}{k}<\lambda <
2^{2k - 1}$. Indeed, it suffices to check the case when $n = 2$,
in which case the vertification follows similarly as in a result of
D'Angelo-Varolin \cite[Theorem 3]{DV04}.
\end{example}


\section*{Appendix: Proof of Theorem
\ref{thm: strictPositivstellensatz} from Catlin-D'Angelo's Theorem}

In this appendix, we sketch how the results of Catlin-D'Angelo
\cite{CD99} can be used to deduce Theorem
\ref{thm: strictPositivstellensatz}. As in the first and second
proofs of Theorem \ref{thm: strictPositivstellensatz}, it suffices to
prove Lemma \ref{weakerstatement}.

Let $\z = (z_1, \ldots, z_n)$. Denote by $\C[\z, \conj{\z}]$ the
complex polynomial
algebra in the indeterminates $z_1, \dots, z_n, \conj{z_1}, \dots,
\conj{z_n}$. Equipped with conjugation, $\C[\z, \conj{\z}]$ has the
structure of a commutative complex $\ast$-algebra. A polynomial
$P \in \C[\z, \conj{\z}]$ is said to be {\emph{Hermitian}} if $P$
equals its conjugate $\conj{P}$. Equivalently, $P$ is Hermitian if
and only if $P(z, \conj{z})$ is real for all $z \in \C^{n}$.
A Hermitian polynomial $P\in \C[\z, \conj{\z}]$ is said to be
\emph{positive on $\C^{n}\setminus\{0\}$} if
$P(z, \conj{z}) > 0$ for all $z \in \C^{n}\setminus\{0\}$.
The {\emph{bidegree}} of a monomial $z^\alpha \conj{z}^{\beta} =
z_1^{\alpha_1} z_2^{\alpha_2} \cdots z_n^{\alpha_n}
\conj{z}_1^{\beta_1} \conj{z}_2^{\beta_2} \cdots
\conj{z}_n^{\beta_n} \in \C[\z,\conj{\z}]$
is $(|\alpha|,|\beta|) = (\alpha_1+ \cdots + \alpha_n,
\beta_1 + \cdots+ \beta_n)$. A {\emph{bihomogeneous polynomial}} is a
complex linear combination of monomials of the same bidegree. If a
bihomogeneous polynomial $P = \sum_{|\alpha| = d, |\beta| = e}
a_{\alpha\beta} z^\alpha \conj{z}^\beta$ is Hermitian, then $d = e$,
i.e. $P$ has bidegree $(d, d)$.

From \cite[Definition 2]{CD99}, a Hermitian bihomogeneous polynomial
$P$ is said to satisfy the \emph{strong global Cauchy-Schwarz} (in
short, \emph{SGCS}) \emph{inequality} if $|P(z, \conj{w})|^2 <
P(z, \conj{z}) P(w, \conj{w})$ whenever $z, w\in \C^{n}$ are linearly
independent.

The following result is a special case of
\cite[Corollary of Theorem 1]{CD99} (where the matrix $M$ of
bihomogeneous polynomials in \cite[Corollary of Theorem 1]{CD99} has
size $1 \times 1$).

\begin{theorem}[Catlin-D'Angelo {\cite[Corollary of Theorem 1]{CD99}}]
\label{thm: CDThm}
Let $R \in \C[\z,\conj{\z}]$ be a nonconstant Hermitian bihomogeneous
polynomial such that $R$ is positive on $\C^n\setminus\{0\}$, the
domain $\{z \in \C^{n} : R(z, \conj{z}) < 1\}$ is strongly
pseudoconvex, and $R$ satisfies the SGCS inequality. Then for each
Hermitian bihomogeneous polynomial $Q \in \C[\z,\conj{\z}]$ positive
on $\C^n \setminus\{0\}$, there exists $l \ge 1$ and polynomials
$h_1,\ldots, h_N \in \C[\z] \subset\C[\z,\conj{\z}]$ such that
$R^l Q = \sum_{k = 1}^N h_k \conj{h_k}$.
\end{theorem}

\begin{proof}[Proof sketch of Lemma \ref{weakerstatement} from
Theorem \ref{thm: CDThm}]
Let $f = \sum_{|\alpha|=d} c_\alpha x^\alpha \in \R[\x]$ be a form of
degree $d$. Suppose that $f$ is nonconstant with strictly positive
coefficients. One verifies that $R := \sum_{|\alpha|=d} c_\alpha
z^\alpha \conj{z}^\alpha$ is a nonconstant Hermitian bihomogeneous
polynomial that is positive on $\C^n\setminus\{0\}$, the domain
$\{z \in \C^{n} : R(z, \conj{z}) < 1\}$ is strongly pseudoconvex, and
that $R$ satisfies the SGCS inequality.

Now suppose that $g = \sum_{|\beta|= e} b_\beta x^\beta \in \R[\x]$
is a form of degree $e$ which is strictly positive on
$\R^n_+\setminus\{0\}$. This implies that $Q := \sum_{|\beta|=e}
b_\beta z^\beta \conj{z}^\beta$ is a Hermitian bihomogenous
polynomial that is positive on $\C^n\setminus\{0\}$. Thus we may
apply Theorem \ref{thm: CDThm} to obtain $l \ge 1$ such that
$R^l Q = \sum_{k = 1}^N h_k \conj{h_k}$ for some polynomials
$h_1,\ldots, h_N \in \C[\z] \subset\C[\z,\conj{\z}]$. Hence
$R^l Q = \sum_{|\alpha| = |\beta| = ld + e} a_{\alpha\beta}
z^\alpha \conj{z}^\beta$ for some positive semidefinite Hermitian
matrix $A = (a_{\alpha\beta})_{|\alpha| = |\beta| = ld + e}$.
Writing $f^lg = \sum_{|\gamma| = ld + e} a_{\gamma}^\prime x^\gamma$,
we see that $A$ is in fact the diagonal matrix
$\mathrm{diag}(a_{\gamma}^\prime)_{|\gamma| = ld + e}$.
Since $A$ is positive semidefinite, all the coefficients
$a_{\gamma}^\prime$ of $f^l g$ are nonnegative. This completes the
proof of Lemma \ref{weakerstatement}.
\end{proof}


\end{document}